\documentclass[11pt,a4paper,reqno]{amsart}
\usepackage{amsmath,amssymb,amsfonts,epsfig,mathrsfs,cite, hyperref}
\usepackage[T1]{fontenc}
\usepackage{color}
\usepackage{array}
\usepackage{amsthm}
\usepackage{amstext}
\usepackage{graphicx}
\usepackage{setspace}

\usepackage[title]{appendix}

\usepackage{booktabs}
\usepackage{longtable}
\usepackage{tcolorbox}

\usepackage{float}

\makeatletter
\@namedef{subjclassname@2020}{%
  \textup{2020} Mathematics Subject Classification}
\makeatother

\usepackage[margin=2.5cm]{geometry}
\usepackage{color}
\usepackage{enumitem}

\usepackage{amscd,psfrag}
\usepackage{yhmath}
\usepackage[mathscr]{eucal}
\usepackage{pdflscape}

\usepackage{comment}

\allowdisplaybreaks[4]

\usepackage{slashed}

\makeatletter
\pdfpageheight\paperheight
\pdfpagewidth\paperwidth

\usepackage{epstopdf}
\usepackage{indentfirst}	

\usepackage[normalem]{ulem}
\theoremstyle{plain}

\newtheorem{definition}{Definition}
\newtheorem{theorem}[definition]{Theorem}
\newtheorem*{theorem*}{Theorem}

\newtheorem*{remark*}{Remark}
\newtheorem*{sideremark*}{Side Remark}

\newtheorem*{claim*}{Claim}
\newtheorem*{lemma*}{Lemma}
\newtheorem*{q*}{Question}
\newtheorem{lemma}[definition]{Lemma}

\newtheorem*{corollary*}{Corollary}

\newcommand{\R}{\mathbb{R}}

\def\XXint#1#2#3{{\setbox0=\hbox{$#1{#2#3}{\int}$ }
\vcenter{\hbox{$#2#3$ }}\kern-.6\wd0}}

\author{Siran Li}

\address{Siran Li: School of Mathematical Sciences $\&$ CMA-Shanghai, Shanghai Jiao Tong University, No.~6 Science Buildings, 800 Dongchuan Road, Minhang District, Shanghai, China (200240)}

\email{\texttt{siran.li@sjtu.edu.cn}}

\author{Isaac Newell}

\address{Isaac Newell: Mathematical Institute and Hertford College, University of Oxford, Andrew Wiles Building, Radcliffe
Observatory Quarter, Woodstock Road, Oxford OX2 6GG, UK.}

\email{\texttt{isaac.newell@hertford.ox.ac.uk}}



\date{\today}

\title{A geometric rigidity estimate for codimension-1 immersions into spheres}

\begin{document}

\begin{abstract}
We establish a quantitative stability estimate for codimension-one immersions into a round sphere. Let $(M,g)$ be an oriented Riemannian manifold and suppose that a prescribed shape operator is realised by a smooth isometric immersion $\theta:(M,g)\to\mathbb S^{n+1}$. We prove that, on every relatively compact strongly Lipschitz domain and for every $1<p<\infty$, any Sobolev immersion $\phi$ is close to $\theta$ modulo an ambient rotation, with the $W^{1,p}$-distances between both the immersions and their Gauss maps controlled by the $L^p$ stretching-plus-bending energies of $\phi$. No {\it a priori} bounds on the fundamental forms of $\phi$ are required. The proof proceeds by extending the immersions along normal geodesics and reducing the problem to an equidimensional geometric rigidity estimate on the sphere. A finite localisation and patching argument handles the possible non-injectivity of the normal extension of the reference immersion.
\end{abstract}

\maketitle

\section{Introduction and main results}

Let $(M,g)$ be a smooth oriented Riemannian manifold of dimension $n$ and denote by $\mathbb{S}^{n+1}$ the round sphere with the canonical metric $\sigma$. A well-known rigidity theorem states that for two isometric immersions  $\theta, \phi : (M,g) \to (\mathbb{S}^{n+1},\sigma)$ with the same second fundamental form,  there exists a proper isometry $f$ of $(\mathbb{S}^{n+1},\sigma)$ such that $\phi = f \circ \theta$. This is the uniqueness part of the fundamental theorem of submanifold theory; see, e.g. \cite[\S3.2]{by_chen}. In this work, we study the stability of this property.

We first recall some terminology and notation on isometric immersions. Let $\theta: M \to \mathbb{S}^{n+1}$ be a smooth immersion. Then, for each $x \in M$, $d\theta_x$ is an injective linear map from $T_xM$ into $T_{\theta(x)}\mathbb{S}^{n+1}$. Moreover, $\theta$ is isometric (denoted as $d\theta \in O(g,\sigma)$) if and only if $d\theta_x : (T_xM, g_x) \to (T_{\theta(x)}\mathbb{S}^{n+1},\sigma_{\theta(x)})$ is an isometry for every $x \in M$. The Gauss map $\nu_\theta : M \to T\mathbb{S}^{n+1}$ is defined so that $\nu_\theta(x) \in T_{\theta(x)}\mathbb{S}^{n+1} = \langle \theta(x) \rangle^\perp \leq \mathbb{R}^{n+2}$ is the unique unit vector orthogonal to $d\theta_x(T_x M)$ for which $(d\theta_x(v_1), \cdots, d\theta_x(v_n), \nu_\theta(x), \theta(x))$ is a positive basis of $\mathbb{R}^{n+2}$ if $(v_1, \cdots, v_n)$ is a positive basis of $T_x M$. The shape operator $S_\theta : TM \to TM$ is defined by
\begin{equation*}
    \nabla_X \nu_\theta = - (d\theta \circ S_\theta)(X), \quad X \in \Gamma(TM)
\end{equation*}
where $\nabla$ is the usual connection on $\mathbb{R}^{n+2}$. (This is well-defined because $\nabla_X\nu_\theta$ is orthogonal to $\theta$ and to $\nu_\theta$ in $\mathbb{R}^{n+2}$, and hence belongs to the range of $d\theta$.) 

Suppose now that $(M,g)$ is equipped with a reference shape operator $S$, which is a smooth, self-adjoint $(1,1)$-tensor field. Let $p \in (1,\infty)$ and suppose that $\omega \subset\subset M$ is a strongly Lipschitz domain compactly contained in $M$ (\emph{cf}. Definition~\ref{def:lip_mfd}). In this work, we consider the following \textit{stretching-plus-bending} energy of an elastic deformation $\theta$,  studied in Alpern--Kupferman--Maor~\cite{akm22_resh_codim1, akm24_stability} and inspired by similar energies appearing in physics literature such as \cite{kes07_shaping, gsd16_nanoribbons, gv11_shape_selection, dhs11_buckling}:
\begin{align*}
E_p(\theta;\omega) &:= \int_\omega \operatorname{dist}^p(d\theta, O(g, \sigma)) dvol_g + \int_\omega |d\theta \circ (S_\theta - S)|_{g,\sigma}^p dvol_g\\
&= E_p^S(\theta; \omega) + E_p^B(\theta;\omega). \nonumber
\end{align*}
Here, the \emph{stretching} energy $E_p^S$ measures how far $\theta$ deviates from being isometric ($\theta$ is isometric if and only if $d\theta \in O(g,\sigma)$). The \emph{bending} energy $E_p^B$ measures how far $S_\theta$ deviates from the prescribed shape operator $S$. Thus, an immersion $\theta :(\omega,g) \to (\mathbb{S}^{n+1}, \sigma)$ has $E_p(\theta;\omega) = 0$ if and only if $\theta$ is isometric and $S_\theta = S$.

 \begin{remark*}
     In the elasticity literature, a central problem is the derivation of dimensionally-reduced theories from nonlinear bulk elasticity. Consider the following formulation for thin shells. Let $S \subset \R^3$ be a compact, embedded surface with unit normal $n$, and let
     \begin{equation*}
         S^h := \{x+t n(x) : x \in S, -\frac h 2 < t < \frac h 2\}.
     \end{equation*}
     For deformations $u \in W^{1,2}(S^h; \R^3)$, the elastic energy per unit thickness is
  \begin{equation*}
         E^h(u) := \frac1h\int_{S^h} W(\nabla u),
     \end{equation*}
     where the stored energy density function $W : \R^{3 \times 3} \to [0,\infty]$ satisfies certain natural assumptions. One then seeks to identify the $\Gamma$-limit $I_\beta$ of the rescaled energies $h^{-\beta} E^h$ as $h \searrow 0$. Different limiting theories arise for different values of $\beta$; see \cite{fjm, fjm2, fjm_mora_03, leDret_raoult_95, leDret_raoult_96, lewicka_pakzad_11, lewicka_mahadevan_pakzad_11, lewicka_mora_pakzad_09, lewicka_mora_pakzad_10, lewicka_mora_pakzad_11, lewicka_mahadevan_pakzad_17, hornung_lewicka_pakzad_13, conti_etal_02, conti_maggi}. In many such cases, $I_\beta$ contains terms which quantify stretching and/or bending in some way, as does our energy $E_p$.
 \end{remark*}

The natural function space for studying the energy $E_p$ is the space of $p$-Sobolev immersions, which pertains to elastic bodies of weak regularity:
\begin{multline*}
    I_p(\omega; \mathbb{S}^{n+1}) := \Big\{ \theta \in W^{1,p}(\omega; \mathbb{R}^{n+2}) : |\theta| = 1,~ \mathrm{rank}( d\theta) = n \text{ a.e., }
    \nu_\theta \in W^{1,p}(\omega; \mathbb{R}^{n+2}) \Big\}.
\end{multline*}

The main result of our note is as follows.
\begin{theorem} \label{thm:rig_main_codim1}
Let $(M,g)$ be a smooth oriented $n$-dimensional Riemannian manifold without boundary, $p \in (1,\infty)$, and $\omega \subset\subset M$ be a strongly Lipschitz domain compactly contained in $M$. Assume that $S$ is a smooth, self-adjoint $(1,1)$-tensor field on $M$ and that there exists a smooth isometric immersion $\theta : (M,g) \to \mathbb{S}^{n+1}$ having shape operator $S$. Then, there exists a constant $C$ depending only on $\theta$, $(M,g)$, $S$, $p$, and $n$, such that for any immersion $\phi \in I_p(\omega;\mathbb{S}^{n+1})$, one can find $Q \in SO(n+2)$ satisfying 
    \begin{equation} \label{eq:rig_main_codim1}
        \|\phi - Q\theta\|_{W^{1,p}(\omega,g; \mathbb{R}^{n+2})}^p  + \| \nu_\phi - Q\nu_\theta\|_{W^{1,p}(\omega, g;\mathbb{R}^{n+2})}^p \leq C E_p(\phi; \omega).
    \end{equation}
\end{theorem}
\begin{remark*}
    Recall that if $M$ is simply connected, the existence of an isometric immersion $\theta: (M,g) \to \mathbb S^{n+1}$ with shape operator $S$ is equivalent to the condition that $(g,S)$ satisfy the Gauss-Codazzi equations.
\end{remark*}

Immersions of elastic bodies into non-flat targets have attracted much attention in the physics community; see, \emph{e.g.}, \cite{aharoni2016internal}. The analogue of Theorem \ref{thm:rig_main_codim1} for codimension-1 immersions into Euclidean spaces has been recently established in Alpern--Kupferman--Maor~\cite[Theorem A.1]{akm24_stability}. Theorem \ref{thm:rig_main_codim1} is, to our knowledge, the first such \textit{quantitative} rigidity result into a non-Euclidean target. In the  nice preprint~\cite[Theorem 3.1]{bastug_quant_local}, Ba\c{s}tu\v{g} obtained a similar estimate for maps from a cube in $\mathbb{R}^n$ to a compact, oriented $(n+1)$-dimensional manifold $N$. The target $N$ in \cite{bastug_quant_local} is fairly general; nevertheless, in comparison with~\eqref{eq:rig_main_codim1}, the estimate in \cite{bastug_quant_local} contains extra terms on the right-hand side, including the $L^p$-norm of $d\phi$ and the oscillation of $g$. Various quantitative rigidity results for surfaces in $\mathbb{R}^3$ have also been established by P. Ciarlet and collaborators in \cite{ciarlet_03, ciarlet_malin_cmardare_19, ciarlet_malin_cmardare_20, ciarlet_cmardare_19}, assuming certain \textit{a priori} bounds on the first and second fundamental forms. Malin--C. Mardare~\cite{cm_korn_hypersurf_18} also proved rigidity results for hypersurfaces in $\mathbb{R}^{n+1}$ with a slightly different stretching-plus-bending energy.

On the other hand, \textit{asymptotic}, rather than quantitative, rigidity results for the energy $E_p$ are already known for more general target spaces~\cite{akm22_resh_codim1, bastug_quant_local, bastug_asymptotic_noncompact, km25_linearization}. In particular, Alpern--Kupferman--Maor~\cite{akm24_stability} established the following result: Let $(M,g)$, $S$, and $p$ be as in Theorem \ref{thm:rig_main_codim1} and let $(N,h)$ be a connected, oriented $(n+1)$-dimensional manifold without boundary. Suppose that $(\theta_k) \subset I_p(M;N)$ is a sequence of immersions for which $E_p(\theta_k) \to 0$. Then there exists a smooth isometric immersion $\theta : (M,g) \to (N,h)$ with shape operator $S$ such that $\theta_k \to \theta$ in $W^{1,p}(M;N)$ and $\nu_{\theta_k} \to \nu_{\theta}$ in $W^{1,p}(M;TN)$. (We omit here the precise definitions of $I_p(M;N)$, $W^{1,p}(M;TN)$, and $E_p$ in the more general setting. See~\cite{akm24_stability} for details.) Hence,  $\inf_{I_p(M;N)} {E_p}$ is zero if and only if there exists a smooth map $\theta$ with $E_p(\theta) = 0$.

The main novelty of our work, especially in comparison with \cite{akm24_stability}, is twofold:
\begin{enumerate}
    \item We treat the non-Euclidean target $\mathbb{S}^{n+1}$ by extending the argument in \cite{akm24_stability}.
    \item We handle the possible non-injectivity of (the extension along normal geodesics in $\mathbb{S}^{n+1}$ of) $\theta$. 
\end{enumerate}

We shall deduce Theorem \ref{thm:rig_main_codim1} from Theorem~\ref{thm:rig_main_codim0}, a codimension-0 rigidity estimate for elastic deformations. It can be regarded as a variant of the \textit{nonlinear Korn inequalities}, which bound the distance between two immersions from $\Omega \subset \mathbb{R}^n$ to $\mathbb{R}^n$ (modulo Euclidean rotations) by the distance between their Cauchy-Green tensors; see, \emph{e.g.}, \cite{ciarlet_cmardare_04, cm_nki_2015, cm_nki_W2p_2016}.

\begin{theorem} \label{thm:rig_main_codim0}
Let $\mathcal{M}$ be an oriented smooth $n$-dimensional manifold without
boundary, $p\in(1,\infty)$, and $\Omega\subset\subset\mathcal{M}$ a strongly Lipschitz domain compactly contained in $\mathcal{M}$. Suppose that
$\Theta\in C^1(\mathcal{M};\mathbb{S}^n)$ is an orientation-preserving immersion. Then there is a constant $C=C(\Theta,\Omega,p,n)<\infty$ such that for every $\Phi\in W^{1,p}(\Omega;\mathbb{S}^n)$, there exists a rotation $Q \in SO(n+1)$ satisfying
\begin{equation} \label{eq:rig_main_codim0}
    \|\Phi-Q\Theta\|_{W^{1,p}(\Omega, \Theta^*e;\mathbb{R}^{n+1})} \leq C\|\operatorname{dist}(d\Phi, SO(\Theta^*\sigma,\sigma))\|_{L^p(\Omega, \Theta^*e)}.
\end{equation}
\end{theorem}

As with \cite{ciarlet_cmardare_04, cm_nki_2015}, our proof of Theorem~\ref{thm:rig_main_codim0} is based on the landmark \textit{geometric rigidity lemma} of Friesecke--James--M\"uller~\cite{fjm}, which enables the rigorous derivation of dimensionally-reduced theories in nonlinear elasticity via Gamma convergence; see \cite{fjm_mora_03, fjm, fjm2, lewicka_pakzad_11, lewicka_mora_pakzad_10}. We use a spherical variant of that geometric rigidity estimate proved recently in Kupferman--Maor~\cite{km25_linearization};  see also \cite{chen_li_slemrod_22, cdm} for some Riemannian generalisations. Another key point in the proof of Theorem \ref{thm:rig_main_codim0} is to cover the domain $\Omega$ by finitely many domains on which $\Theta$ is injective. Then, we can apply the geometric rigidity estimate on each domain and perform a change of variables on each via $\Theta$. Finally, we argue that the resulting local inequalities can be patched together to recover Theorem \ref{thm:rig_main_codim0}.

For the passage from Theorem \ref{thm:rig_main_codim0} to Theorem \ref{thm:rig_main_codim1}, we follow a well-known strategy to extend the given codimension-1 immersions $\theta, \phi : \omega \subset M \to \mathbb{S}^{n+1}$ to $\Omega := \omega \times (-h,h)$ for some small $h > 0$ by `thickening' along normal geodesics:
\begin{equation*}
    \Theta(x,t) := \operatorname{exp}^{\mathbb{S}^{n+1}}_{\theta(x)}(t\nu_\theta(x)) = \cos t\ \theta(x) + \sin t\ \nu_\theta(x)
\end{equation*}
with $\Phi$ extending $\phi$ similarly. We then apply the codimension-0 rigidity Theorem~\ref{thm:rig_main_codim0} to $\Theta$ and $\Phi$, which yields \eqref{eq:rig_main_codim0}. The remaining challenge is to estimate the left-hand side from below and the right-hand side from above, by correct quantities depending only on the original maps $\theta, \phi$. This argument is similar to that used in various papers by the Ciarlet group \cite{ciarlet_03, ciarlet_malin_cmardare_19, ciarlet_malin_cmardare_20, ciarlet_cmardare_19, cm_korn_hypersurf_18} and in the Euclidean counterpart of Theorem \ref{thm:rig_main_codim1} from \cite{akm24_stability}. However, the non-Euclidean ambient space forces us to make some adaptations.

The remainder of this note is organised as follows. In \S \ref{sec:prelims_GeomRig}, we recall some preliminaries. In \S \ref{sec:codim0_proof} we prove Theorem \ref{thm:rig_main_codim0}. Finally, in \S \ref{sec:codim1_proof} we prove our main Theorem \ref{thm:rig_main_codim1}.

\section{Preliminaries} \label{sec:prelims_GeomRig}
We briefly introduce some notation and recall some relevant facts about linear algebra and about Lipschitz domains on Riemannian manifolds. 

\subsection{Notation}
We shall write $e$ for the Euclidean metric on $\R^D$ for any dimension $D$, $\iota : \mathbb S^n \to \R^{n+1}$ for the standard embedding, and $\sigma := \iota^*e$ for the canonical metric on $\mathbb S^n$. For $1 \leq p < \infty$ and a compact manifold $M$ (possibly with boundary) with a Riemannian metric $g$, we will write $\|\cdot\|_{L^p(M;g)}$ for the $L^p$ norm with respect to $g$:
\begin{equation*}
    \|u\|_{L^p(M,g)} := \Big(\int_M |u(x)|^p dvol_g(x) \Big)^{1/p}
\end{equation*}
and $\|\cdot\|_{W^{1,p}(M,g)}$ for the $W^{1,p}$ norm with respect to $g$:
\begin{equation*}
    \|u\|_{W^{1,p}(M,g)} := \Big( \int_M |u(x)|^pdvol_g(x) + \int_M |du_x|_{g,e}^p dvol_g(x) \Big)^{1/p}.
\end{equation*}
Since $M$ is compact, the $L^p$ and $W^{1,p}$ norms induced by any two continuous metrics $g$, $\tilde g$ are equivalent.

\subsection{Linear algebra}
Let $V$ and $W$ be finite dimensional real inner product spaces with inner products $g$ and $h$ and dimensions $n = \dim V$, $m = \dim W$. We denote by $\operatorname{Hom}(V,W)$ the space of linear maps from $V$ to $W$ and write $O(V,W)$ for the set of linear isometries from $V$ to $W$, which is nonempty if and only if $n \leq m$. If furthermore $V$ and $W$ are oriented and $n=m$, then we denote by $SO(V,W)$ the set of orientation-preserving linear isometries from $V$ to $W$. The inner products $g$, $h$ induce an inner product on $\operatorname{Hom}(V,W)$. The resulting norm of a map $T \in \operatorname{Hom}(V,W)$ is given by
\begin{equation*}
    |T|_{g,h} := \Big( \sum_{i=1}^n |Tv_i|_h^2 \Big)^{1/2}
\end{equation*}
for any choice of orthonormal basis $(v_i)_{i=1}^n$ of $(V,g)$; this is well-defined and independent of the choice of such basis. We refer to $|\cdot|_{g,h}$ as the Frobenius norm. Moreover, for a compact set $K \subset \operatorname{Hom}(V,W)$ (typically $O(V,W)$ or $SO(V,W)$) we will write
\begin{equation*}
    \operatorname{dist}(T,K) := \min_{Q \in K} |T-Q|_{g,h}.
\end{equation*}
Recall that if $n \leq m$ and $T \in \operatorname{Hom}(V,W)$, then $T$ can be factorised as
\begin{equation*}
    T = QS
\end{equation*}
where $Q \in O(V,W)$ and $S \in \operatorname{Hom}(V)$ is self-adjoint with non-negative eigenvalues. If $T$ is injective, then this factorisation is unique, and given by
\begin{equation*}
    Q = T(T^*T)^{-\frac12}, \quad S = (T^*T)^\frac12.
\end{equation*}
Here, $T^* \in \operatorname{Hom}(W,V)$ denotes the adjoint map satisfying $\langle Tv, w\rangle_h = \langle v, T^*w\rangle_g$ for all $v \in V$, $w \in W$, and $(T^*T)^\frac12$ is the unique self-adjoint map $M$ on $V$ with positive eigenvalues satisfying $M^2 = T^*T$. Moreover, there holds
\begin{equation} \label{eq:closest_O}
    \min_{R \in O(V,W)} |T-R|_{g,h} = |T-Q|_{g,h} = |I_V-S|_{g,g},
\end{equation}
so $Q$ is the closest element to $T$ in $O(V,W)$. In particular, if $n=m$ and $V$ and $W$ are oriented, then for any injective orientation-preserving map $T \in \operatorname{Hom}(V,W)$, $Q := T(T^*T)^{-\frac12}$ belongs to $SO(V,W)$ and is the closest such element to $T$.

In our applications of interest, the oriented inner product spaces at hand will be the tangent spaces of Riemannian manifolds. Suppose that $(M,g)$ and $(N,h)$ are compact Riemannian manifolds (with $M$ possibly having a boundary). Choose a smooth isometric embedding $\kappa : (N,h) \hookrightarrow \mathbb{R}^D$ via the Nash embedding theorem. Then, considering the Sobolev space $W^{1,p}(M;N)$ extrinsically, i.e. as the space of maps $u \in W^{1,p}(M;\mathbb{R}^D)$ such that $\kappa \circ u \in W^{1,p}(M;\mathbb{R}^D)$, the differential of $\kappa \circ u$ is for almost every $x \in M$ a linear map from $T_xM$ into $T_{u(x)}N$. We shall write $\operatorname{dist}(du, O(g,h))$ to denote the measurable function on $M$ given at a.e. $x$ by
\begin{equation*}
    \operatorname{dist}(du_x, O((T_xM, g_x), (T_{u(x)}N, h_{u(x)}))).
\end{equation*}

\subsection{Lipschitz domains on manifolds}

\begin{definition} \label{def:lip_Rn}
    Let $\Omega \subset \mathbb{R}^n$ be a connected open set. We say that $\Omega$ is a \emph{strongly Lipschitz domain near} a point $x_0 \in \partial\Omega$ if there exists an affine $(n-1)$-dimensional hyperplane $H$ with unit normal $\nu$, numbers $r,h > 0$, and an open cylinder 
    \begin{equation} \label{eq:def_Crh}
        C_{r,h} := \{x' + t\nu : x' \in H, |x'-x_0|< r, |t|<h\}
    \end{equation}
    such that
    \begin{equation}
        C_{r,h} \cap \Omega = C_{r,h} \cap \{x' +t\nu : x' \in H, t > \gamma(x')\}
    \end{equation}
    for some Lipschitz function $\gamma : H \to \mathbb{R}$ satisfying
    \begin{equation}
        \gamma(x_0) = 0 \quad \text{and} \quad |\gamma(x')| < h \text{ for all } |x'|\leq r.
    \end{equation}
    Furthermore, we say that $\Omega$ is a \emph{locally strongly Lipschitz domain} if it is a strongly Lipschitz domain near every point $x_0 \in \partial\Omega$. A locally strongly Lipschitz domain with compact boundary is simply called a \emph {strongly Lipschitz domain}.
\end{definition}

\begin{definition} \label{def:lip_mfd}
    Let $\mathcal{M}$ be an $n$-dimensional manifold (without boundary) with a $C^1$ atlas $\mathcal{A}$. We say that a connected open set $\Omega \subset \mathcal{M}$ is a \emph{locally strongly Lipschitz domain relative to} $\mathcal{A}$ if for every $x_0 \in \partial\Omega$ there is a local chart $(U,\varphi) \in \mathcal{A}$ with $x_0 \in U$ such that $\varphi(U \cap \Omega) \subset \mathbb{R}^n$ is a locally strongly Lipschitz domain near $\varphi(x_0)$. If $\partial\Omega$ is compact, we call $\Omega$ a \emph{strongly Lipschitz domain}.
\end{definition}

We are chiefly interested in the setting where $\mathcal{M}$ is a smooth manifold equipped with a chosen smooth structure and $\mathcal{A}$ is the compatible maximal $C^1$ atlas. Hence, we shall not explicitly mention the atlas $\mathcal{A}$ in the sequel. We shall also restrict our attention to the case that the closure of $\Omega$ in $\mathcal{M}$ is compact, thus excluding the possibility of $\Omega$ `going off the edge' of the ambient manifold $\mathcal{M}$, which we do not assume to be complete.  By \cite[Corollary 4.2]{HMT}, Definitions~\ref{def:lip_Rn} and~\ref{def:lip_mfd} are consistent when $\mathcal{M} = \mathbb{R}^n$. We also have the following lemma; see \cite[Theorem 4.1]{HMT} for the Euclidean case.
\begin{lemma} \label{lemma:C1diffeo_takes_lip_to_lip}
    Let $\mathcal{M}$ and $\mathcal{N}$ be smooth $n$-dimensional manifolds. Suppose that $\mathcal{O} \subset \mathcal{M}$ is an open set and that $F : \mathcal{O} \to \mathcal{N}$ is a $C^1$ diffeomorphism onto its image. If $\Omega \subset \mathcal{M}$ is a strongly Lipschitz domain with compact closure contained in $\mathcal{O}$, then $F(\Omega)$ is a strongly Lipschitz domain in $\mathcal{N}$.
\end{lemma}

\section{Proof of Theorem \texorpdfstring{\ref{thm:rig_main_codim0}}{2}} \label{sec:codim0_proof}

In this section, we prove the codimension-0 rigidity estimate of the nonlinear Korn type, namely Theorem~\ref{thm:rig_main_codim0}. For this purpose, we first decompose the domain $\Omega$ into finitely many Lipschitz subdomains on which $\Theta$ is injective and bi-Lipschitz.

\begin{lemma}
\label{lem:cover}
Under the hypotheses of Theorem~\ref{thm:rig_main_codim0}, there are finitely many strongly Lipschitz domains $U_1,\dots,U_N \subset \Omega$  such that:
\begin{enumerate}
\item[(i)] $\Omega=\bigcup_{i=1}^N U_i$;
\item[(ii)] For each $i$, $\Theta|_{U_i}$ is a $C^1$ and bi-Lipschitz diffeomorphism onto its image;
\item[(iii)] The image sets $\widehat U_i:=\Theta(U_i)$ are strongly Lipschitz domains in $\mathbb{S}^n$.
\end{enumerate}
\end{lemma}

\begin{proof}
Since $d\Theta_x:T_x\mathcal{M}\to T_{\Theta(x)}\mathbb{S}^n$ is an isomorphism for every $x \in \mathcal{M}$, the inverse function theorem implies that every point $x \in \overline\Omega$ has an open neighbourhood $\mathcal{O}_x \subset \mathcal{M}$ such that $\Theta|_{\mathcal{O}_x}$ is a $C^1$ diffeomorphism onto its image.  

If $x\in\Omega$, choose $U_x$ to be a geodesic ball centred at $x$ with sufficiently small radius that $U_x$ is compactly contained in $\mathcal{O}_x$ and $U_x$ has smooth (in particular strongly Lipschitz) boundary.

If $x \in \partial\Omega$, then pick a $C^1$ chart $(V, \psi)$ at $x$ such that $\psi(V \cap \Omega) \subset \mathbb{R}^n$ is a locally strongly Lipschitz domain near $\psi(x)$. Hence, we can find an $(n-1)$-dimensional hyperplane $H \subset \mathbb{R}^n$ with unit normal $\nu$, positive numbers $r,h > 0$, and an open cylinder $C_{r,h} \subset \mathbb{R}^n$ as in \eqref{eq:def_Crh} such that 
\begin{equation*}
    C_{r,h} \cap \psi(V \cap \Omega) = C_{r,h} \cap \{x'+t\nu: x' \in H, t > \gamma(x')\}
\end{equation*}
with a Lipschitz function $\gamma: H \to \mathbb{R}$ satisfying $\gamma(0) = 0$ and $|\gamma(x')| <h$ for $|x'|\leq r$. The same conditions still hold with any $(\rho,\eta)$ replacing $(r,h)$ as long as $0<\rho\leq r$, $0<\eta\leq h$, and $\rho {\rm Lip}(\gamma)<\eta$. Hence, by shrinking $r$ and $h$ if necessary we may assume that the cylinder $C_{r,h}$ is compactly contained inside $\psi(V \cap \mathcal{O}_x)$. Let $W_x := \psi^{-1}(C_{r,h})$, which is a connected open neighbourhood of $x$ compactly contained in $\mathcal{O}_x$. Moreover, $U_x := W_x \cap \Omega$ is by construction a connected open subset of $\Omega$ with strongly Lipschitz boundary.

Since the sets $\{U_x\}_{x \in \Omega} \cup \{W_x\}_{x \in \partial\Omega}$ form an open cover for the compact set $\overline\Omega$, we can find a finite subcover $\{U_{x_i}\}_{i=1}^M \cup \{W_{x_i}\}_{i=M+1}^N$. The sets $U_i := U_{x_i}$ for $1\leq i \leq M$ and $U_i := W_{x_i}\cap \Omega$ for $M+1\leq i\leq N$ have the required properties, by Lemma \ref{lemma:C1diffeo_takes_lip_to_lip}.
\end{proof}

\begin{lemma}
\label{lem:pullback}
Let $U\subset\Omega$ be one of the domains $\{U_1,\cdots,U_N\}$ in
Lemma~\ref{lem:cover}, let $\hat U :=\Theta(U)$, and let
$\hat\Theta:=\hat U \to U$ be inverse to $\Theta|_U$. For
$\Phi\in W^{1,p}(U;\mathbb{S}^n)$, define
\begin{equation*}
    u := \Phi \circ \hat\Theta : \hat U \to \mathbb{S}^n.
\end{equation*}
Then $u \in W^{1,p}(\hat U;\mathbb{S}^n)$ and for every $Q \in SO(n+1)$, there holds
\begin{align}
 \|u-Q\iota\|_{W^{1,p}(\hat U,\sigma;\mathbb{R}^{n+1})} &= \|\Phi-Q\Theta\|_{W^{1,p}(U, \Theta^*e; \mathbb{R}^{n+1})}, \label{eq:pullback-W}\\
 \|\operatorname{dist}(du,SO(\sigma,\sigma))\|_{L^p(\hat U,\sigma)} &= \|\operatorname{dist}(d\Phi, SO(\Theta^*e, \sigma))\|_{L^p(U,\Theta^*e)}. \label{eq:pullback-strain}
\end{align}
\end{lemma}

\begin{proof}
By Lemma \ref{lem:cover}, $\Theta|_{U}$ is a bi-Lipschitz and $C^1$ diffeomorphism onto the strongly Lipschitz domain $\hat U$. Hence, the composition $u = \Phi \circ \hat\Theta$ belongs to $W^{1,p}(\hat U; \mathbb{S}^n)$ and its differential satisfies the chain rule:
\begin{equation} \label{eq:chain_u}
    du_{\Theta(x)} = d\Phi_x \circ (d\Theta_x)^{-1}
\end{equation}
for a.e. $x \in U$. (This fact can be deduced from the corresponding Euclidean statement by working in local charts.) By the definition of the pullback metric $\Theta^*e$, the map
\begin{equation*}
    d\Theta_x: (T_x\mathcal{M}, (\Theta^*e)_x) \to (T_{\Theta(x)}\mathbb{S}^n,\sigma_{\Theta(x)})
\end{equation*}
is a linear isometry for every $x \in \mathcal{M}$, and is orientation-preserving by our assumption on $\Theta$. Moreover, for a.e. $x \in \Omega$ we have
\begin{equation*}
     d(u-Q\iota)_{\Theta(x)} = (d\Phi_x - Q d\Theta_x) \circ (d\Theta_x)^{-1}.
\end{equation*}
Since composition with the isometry $(d\Theta_x)^{-1}$ preserves the
Frobenius norm, we have
\begin{equation*}
    |d(u-Q\iota)_{\Theta(x)}|_{\sigma_{\Theta(x)}, e} = |d\Phi_x - Qd\Theta_x|_{(\Theta^*\sigma)_x, e}.
\end{equation*}
Moreover, by definition of $u$ we have $(u-Q\iota)(\Theta(x)) =\Phi(x)-Q\Theta(x)$. Hence, a change of variables via $\hat\Theta$ yields \eqref{eq:pullback-W}.

We now prove \eqref{eq:pullback-strain}. Recall that if $V,W,Z$ are oriented
inner-product spaces of equal dimension, $A \in \operatorname{Hom}(V,Z)$ is invertible and orientation-preserving, and $B \in SO(V,W)$, then \eqref{eq:closest_O} implies that
\begin{equation*}
 \operatorname{dist}\bigl(A\circ B^{-1},SO(W,Z)\bigr)
 =\operatorname{dist}\bigl(A,SO(V,Z)\bigr).
\end{equation*}
Applying this with $A=d\Phi_x$ and $B = d\Theta_x$, and using \eqref{eq:chain_u}, we obtain \eqref{eq:pullback-strain}.
\end{proof}

\begin{lemma}
\label{lem:matrix-coercivity}
Let $E\subset\mathbb{S}^n$ be a nonempty open set and let $p\in[1,\infty)$. There is a constant $c = c(E,p,n)>0$ such that
\begin{equation}
\label{eq:matrix-coercivity}
 \Big( \int_E |Ay|^p dvol_\sigma(y) \Big)^{1/p} \ge c|A|
 \quad \forall A\in\mathbb{R}^{(n+1)\times(n+1)},
\end{equation}
where $|A|$ denotes the Frobenius norm.
\end{lemma}

\begin{proof}
It is clear that the left-hand side (as a function of $A$) defines a seminorm on $\mathbb{R}^{(n+1)\times(n+1)}$. In fact, it is a norm. To check the positivity condition, if the left-hand side of \eqref{eq:matrix-coercivity} is zero, then $Ay=0$ for every $y\in E$. Because any nonempty open subset of $\mathbb{S}^n$ spans $\mathbb{R}^{n+1}$, this implies that $A=0$. All norms on a finite-dimensional vector space are equivalent, which gives \eqref{eq:matrix-coercivity}. 
\end{proof}

We are now ready to give the proof of Theorem \ref{thm:rig_main_codim0}.

\begin{proof}[Proof of Theorem~\ref{thm:rig_main_codim0}]
Choose the finite cover $U_1,\dots,U_N$ from
Lemma~\ref{lem:cover}, and put $\hat U_i := \Theta(U_i)$. For each $i$, define
\begin{equation*}
    u_i:=\Phi\circ(\Theta|_{U_i})^{-1}: \hat U_i \to \mathbb{S}^n.
\end{equation*}
By Lemma~\ref{lem:pullback}, $u_i\in W^{1,p}(\hat U_i;\mathbb{S}^n)$. Applying the spherical geometric rigidity estimate in Kupferman--Maor \cite[Theorem 4.2]{km25_linearization} on the strongly Lipschitz domain $\hat U_i\subset\mathbb{S}^n$, we obtain rotations $Q_i \in SO(n+1)$ such that
\begin{equation}
\label{eq:km25_linearisation-local}
 \|u_i-Q_i\iota\|_{W^{1,p}(\hat U_i, \sigma; \mathbb{R}^{n+1})} \leq C_i \|\operatorname{dist}(du_i, SO(\sigma,\sigma))\|_{L^p(\hat U_i,\sigma)},
\end{equation}
where the constants $C_i = C_i(\hat U_i, n,p)$ do not depend on $\Phi$. Pulling back these norms onto $\mathcal{M}$ via Lemma~\ref{lem:pullback} gives the inequalities
\begin{equation}
\label{eq:local-estimate}
 \|\Phi-Q_i\Theta\|_{W^{1,p}(U_i, \Theta^*e;\mathbb{R}^{n+1})}
 \le C_i\|e_\Phi\|_{L^p(U_i,\Theta^*e)},
\end{equation}
where for notational convenience we write
\begin{equation*}
    e_\Phi(x) := \operatorname{dist}(d\Phi_x, SO((\Theta^*e)_x, \sigma_{\Phi(x)})), \quad x \in \Omega.
\end{equation*}

We would like to merely sum the inequalities \eqref{eq:local-estimate} over $i=1,\dots,N$ but the rotations $Q_i$ chosen on each patch $U_i$ may be different. We shall show now that, at the price of enlarging the constants by a finite, positive factor still independent of $\Phi$, the choice of $Q := Q_1$ works for Theorem \ref{thm:rig_main_codim0}.

To that end, we compare the rotations on overlapping patches. We will show that in fact, each $Q_i$ is close to $Q_1$ with distance controlled by $\|e_\Phi\|_{L^p(\Omega, \Theta^*e)}$. If $U_i\cap U_j \neq \emptyset$, then $\Theta(U_i \cap U_j)$ is a nonempty open subset of $\mathbb{S}^n$. Since $\Theta|_{U_i}$ is injective, a change of variables and the triangle inequality give
\begin{align*}
 \|(Q_i-Q_j)y\|&_{L^p(\Theta(U_i \cap U_j),\sigma; \mathbb{R}^{n+1})}
 =\|(Q_i-Q_j)\Theta\|_{L^p(U_i\cap U_j, \Theta^*e; \mathbb{R}^{n+1})}\\
 &\leq \|\Phi-Q_i\Theta\|_{L^p(U_i\cap U_j, \Theta^*e; \mathbb{R}^{n+1})}
 +\|\Phi-Q_j\Theta\|_{L^p(U_i\cap U_j,\Theta^*e; \mathbb{R}^{n+1})}\\
 &\leq (C_i+C_j) \|e_\Phi\|_{L^p(\Omega; \Theta^*e)}.
\end{align*}
Estimating the left-hand side from below, Lemma~\ref{lem:matrix-coercivity} implies that there exists a constant $C_* = C_*(\Omega, U_1,\dots, U_N, \Theta, n,p)$ such that
\begin{equation}
\label{eq:edge-rotations}
 |Q_i-Q_j|\leq C_* \|e_\Phi\|_{L^p(\Omega; \Theta^*e)}
 \quad \text{whenever } U_i\cap U_j \neq \emptyset.
\end{equation}
Consider the graph $G$ with vertices $1,\dots, N$ and edges $\{i,j\}$ if and only if $U_i \cap U_j~\neq~\emptyset$. Because $\Omega$ is connected and $\Omega = \cup_{i=1}^N U_i$, $G$ is a connected graph. Hence, by iterating the estimate \eqref{eq:edge-rotations} over a path from $i$ to $1$ in $G$ and enlarging $C_*$ by a factor of $N$, we have
\begin{equation} \label{eq:Qi_Q1_close}
    |Q_i - Q_1| \leq C_* \|e_\Phi\|_{L^p(\Omega; \Theta^*e)} \quad \forall i=1,\dots,N.
\end{equation}

We sum over $i$ and combine \eqref{eq:Qi_Q1_close} with the $N$ local estimates \eqref{eq:local-estimate} to obtain
\begin{align*}
 \|\Phi-Q_1\Theta\|&_{W^{1,p}(\Omega, \Theta^*e;\mathbb{R}^{n+1})}^p \leq  \sum_{i=1}^N \|\Phi-Q_1\Theta\|_{W^{1,p}(U_i, \Theta^*e;\mathbb{R}^{n+1})}^p\\
 &{\leq 2^{p-1}} \sum_{i=1}^N \Big( \|\Phi-Q_i\Theta\|_{W^{1,p}(U_i, \Theta^*e;\mathbb{R}^{n+1})}^p + |Q_i-Q_1|^p \|\Theta\|_{W^{1,p}(U_i,\Theta^*e; \mathbb{R}^{n+1})}^p \Big)\\
 &{\leq 2^{p-1}} \sum_{i=1}^N \Big( C_i^p \|e_\Phi\|_{L^p(\Omega, \Theta^*e)}^p + C_*^p \|\Theta\|_{W^{1,p}(\Omega, \Theta^*e; \mathbb{R}^{n+1})}\|e_\Phi\|^p_{L^p(\Omega, \Theta^*e)} \Big)\\
 &= C^p \|e_\Phi\|_{L^p(\Omega, \Theta^*e)}^p
\end{align*}
where $C$ depends on $\Theta, \Omega, n, p$ but not on $\Phi$. Theorem \ref{thm:rig_main_codim0} follows with $Q := Q_1$.
\end{proof}

\section{Proof of Theorem \texorpdfstring{\ref{thm:rig_main_codim1}}{1}} \label{sec:codim1_proof}

In this section we prove our main Theorem~\ref{thm:rig_main_codim1}, \emph{i.e.}, the geometric rigidity estimate for codimension-1 isometric immersions into round spheres. Our arguments largely follow Alpern--Kupferman--Maor~\cite{akm24_stability}.

\begin{proof}[Proof of Theorem \ref{thm:rig_main_codim1}]
    Consider the following map $\Theta : M \times \mathbb{R} \to \mathbb{S}^{n+1}$ defined by extending $\theta$ along normal geodesics in $\mathbb{S}^{n+1}$:
    \begin{equation*}
        \Theta(x,t) := \cos t \theta(x) + \sin t \nu_\theta(x).
    \end{equation*}
    We claim that $\Theta$ is an orientation-preserving immersion on $\mathcal{M} := M \times (-2h,2h)$ for sufficiently small $h > 0$. 
    
    To see this, let us view
    \begin{equation*}
        d\Theta_{(x,t)} : T_xM \times \mathbb{R} \longrightarrow T_{\Theta(x,t)}\mathbb{S}^{n+1} = \langle\Theta(x,t)\rangle^\perp \leq \mathbb{R}^{n+2},
    \end{equation*} 
where $T_{(x,t)} (M \times \mathbb{R}) \simeq T_xM \times \mathbb{R}$. Accordingly, we shall use the following notational convention: given a linear map $A \in \operatorname{Hom}(T_x M; \mathbb{R}^{n+2})$ and a vector $v \in \mathbb{R}^{n+2}$, we will write $A \oplus v$ for the following linear map from $T_x M \times \mathbb{R}$ into $\mathbb{R}^{n+2}$:
    \begin{equation*}
        (A\oplus v)(X,s) := AX +sv, \quad X \in T_xM, s \in \mathbb{R}.
    \end{equation*}
    With this notation, we easily compute that
    \begin{align*}
        d\Theta_{(x,t)} &= (\cos t ~ d\theta_x + \sin t ~ d\nu_\theta(x)) \oplus (-\sin t ~ \theta(x) + \cos t ~ \nu_\theta(x))\\
        &= \big(d\theta_x \circ (\cos t ~ id - \sin t ~ S)\big) \oplus (-\sin t ~ \theta(x) + \cos t ~ \nu_\theta(x))
    \end{align*}
    Since $d\theta_x$ is an invertible linear map from $T_xM$ into the orthogonal complement of the span of $\{\theta, \nu_\theta\}$ in $\mathbb{R}^{n+2}$, we see that $d\Theta_{(x,t)}$ is invertible if and only if $(\cos t ~id - \sin t~S)$ is an invertible linear map on $T_xM$. Replacing $M$ if necessary by an open subset $\Omega' \subset\subset M$ compactly containing $\Omega$, we can assume that $S$ is uniformly bounded over $M$. Because $\sin t \approx 0$ and $\cos t \approx 1$ for small $t$ it is then clear that $\Theta$ is an immersion on $\mathcal{M} := M \times (-2h,2h)$ for sufficiently small $h > 0$. Moreover, the definition of $\nu_\theta$ implies that $\Theta$ is orientation-preserving.

    We can also calculate the pullback metric $\Theta^*\sigma = \Theta^*e$. For any $(x,t) \in \Omega$, $X_1,X_2 \in T_xM$, and $s_1,s_2 \in \mathbb{R}$, we have
    \begin{align}
        \langle (X_1,s_1), (X_2,s_2)&\rangle_{(\Theta^*\sigma)_{(x,t)}} = \big\langle d\Theta_{(x,t)}(X_1,s_1), d\Theta_{(x,t)}(X_2,s_2) \big\rangle_{e} \nonumber \\
        &= \big\langle d\theta_x((\cos t ~id - \sin t ~S)X_1) + s_1 (-\sin t~\theta(x)+\cos t~\nu_\theta(x)), \nonumber \\
        &\qquad d\theta_x((\cos t ~id - \sin t ~S)X_2) + s_2 (-\sin t~\theta(x)+\cos t~\nu_\theta(x)) \big\rangle_e \nonumber \\
        &= \big\langle (\cos t ~id - \sin t ~S)X_1, (\cos t ~id - \sin t ~S)X_2 \big\rangle_{g_x} + s_1 s_2.\label{eq:G_pullback}
    \end{align}
    The last line follows from (a) $\theta$ is an isometric immersion of $g$; (b) the $d\theta_x$-terms are $e$-orthogonal to the terms involving $s_1$ and $s_2$; and (c) the vector $-\sin t~ \theta(x) + \cos t ~\nu_\theta(x)$ has unit length in $\mathbb{R}^{n+2}$.

    Similarly, we can extend $\phi$ along normal geodesics in $\mathbb{S}^{n+1}$. Let $\Omega := \omega \times (-h,h)$, which is a strongly Lipschitz domain compactly contained in $\mathcal{M}$. Define
    \begin{equation*}
        \Phi(x,t) := \cos t \phi(x) + \sin t \nu_\phi(x), \quad (x,t) \in \Omega.
    \end{equation*}
    Then, $\Phi$ belongs to $W^{1,p}(\Omega;\mathbb{S}^{n+1})$. 

    Applying Theorem~\ref{thm:rig_main_codim0} to $\Theta$ and $\Phi$ on $\Omega \subset \mathcal{M}$, we deduce that there exists $Q = Q(\Theta,\Phi) \in SO(n+2)$ such that
    \begin{equation} \label{eq:codim0_prelim}
        \|\Phi - Q\Theta\|_{W^{1,p}(\Omega,\Theta^*\sigma; \mathbb{R}^{n+1})} \leq C \|\operatorname{dist}(d\Phi, SO(\Theta^*\sigma, \sigma))\|_{L^p(\Omega, \Theta^*\sigma)},
    \end{equation}
    where $C = C(\Theta, \Omega, n,p)$. We need to estimate the right-hand side from above and the left-hand side from below by quantities depending only on the original maps $\theta, \phi$.

    We shall argue by orthogonalising the differential $d\Phi$. By a similar calculation to the one for $\Theta$, $d\Phi_{(x,t)} : T_x M \times \mathbb{R} \to \langle\Phi(x,t)\rangle^\perp$ is the linear map given by
    \begin{equation*}
        d\Phi_{(x,t)} = \big(d\phi_x \circ (\cos t ~ id - \sin t ~ S_\phi)\big) \oplus (-\sin t ~ \phi(x) + \cos t ~ \nu_\phi(x))
    \end{equation*}
    for almost every $(x,t) \in \Omega$. Consider the following linear maps:
    \begin{align*}
        &A_{(x,t)} : T_x M \times \mathbb{R} \to \langle \Phi(x,t)\rangle^\perp \leq \mathbb{R}^{n+2},\\
        &A_{(x,t)} := \big( O(d\phi_x) \circ (\cos t ~ id - \sin t ~ S) \big) \oplus (-\sin t ~ \phi(x) + \cos t ~ \nu_\phi(x)),
    \end{align*}
    where $O(d\phi_x) \in O((T_xM, g_x), (T_{\phi(x)}\mathbb{S}^{n+1}, \sigma_{\phi(x)}))$ is the orthogonal map closest to $d\phi_x$ in the Frobenius norm. It is clear that $A$ indeed maps into $\langle \Phi\rangle^\perp$. By an argument similar to the computation of $d\Theta_{(x,t)}$ and the pullback metric $\Theta^*\sigma$, $A_{(x,t)}$ is an orientation-preserving  invertible linear map from $T_xM \times \mathbb{R}$ into $\langle \Phi(x,t) \rangle^\perp$. Moreover, by a calculation analogous to that for \eqref{eq:G_pullback}, we find that $A_{(x,t)}$ is an isometry of inner product spaces from $(T_xM \times \mathbb{R}, (\Theta^*\sigma)_{(x,t)})$ into $(\langle \Phi(x,t)\rangle^\perp, e)$ for a.e. $(x,t) \in \Omega$; namely that  \begin{equation*}
        A \in SO(\Theta^*\sigma, \sigma).
    \end{equation*}
    Therefore, the right-hand side of \eqref{eq:codim0_prelim} can be bounded above by
    \begin{equation} \label{eq:rhs_bd_by_A}
        \operatorname{dist}(d\Phi_{(x,t)}, SO(\Theta^*\sigma,\sigma)) \leq |d\Phi_{(x,t)} - A_{(x,t)}|_{\Theta^*\sigma,\sigma}
    \end{equation}

    Let $G := g + dt^2$ denote the product metric on $\mathcal{M} = M \times (-2h,2h)$. Because $\Theta$ is smooth on $\mathcal{M}$ and $\Omega$ is compactly contained in $\mathcal{M}$, the two metrics $\Theta^*\sigma$ and $G$ are equivalent on $\overline\Omega$, i.e. there exists a constant $c>0$ depending only on $\Theta$ and $\Omega$ such that
    \begin{equation*}
        c |(X,s)|_{G_{(x,t)}} \leq |(X,s)|_{(\Theta^*\sigma)_{(x,t)}} \leq c^{-1}|(X,s)|_{G_{(x,t)}} \quad \forall (x,t) \in \overline\Omega, ~(X,s) \in T_{(x,t)}\mathcal{M}.
    \end{equation*}
    Therefore, for a.e. $(x,t) \in \Omega$ there holds
    \begin{equation} \label{eq:Thetasigma_vs_G}
        |d\Phi_{(x,t)} - A_{(x,t)}|_{\Theta^*\sigma,\sigma} \leq C |d\Phi_{(x,t)} - A_{(x,t)}|_{G,\sigma}
    \end{equation}
    where $C$ only depends on $\Theta$ and $\Omega$. Because $A_{(x,t)}$ and $d\Phi_{(x,t)}$ agree on the $\mathbb{R}$-factor of $T_{(x,t)}\mathcal{M} \simeq T_xM \times \mathbb{R}$, we have
    \begin{align*}
        |d\Phi_{(x,t)} - A_{(x,t)}|_{G,\sigma} &= \big|d\phi_x \circ (\cos t ~id - \sin t ~S_\phi) - O(d\phi_x) \circ (\cos t ~id - \sin t~ S) \big|_{g,\sigma}\\
        &\leq \big|(d\phi_x-O(d\phi_x)) \circ(\cos t~id-\sin t~S)\big|_{g,\sigma} + |\sin t| \big|d\phi_x\circ(S-S_\phi)\big|_{g,\sigma}\\
        &\leq C \big( \operatorname{dist}(d\phi_x, O(g,\sigma)) + |d\phi_x \circ (S-S_\phi)|_{g,\sigma} \big)
    \end{align*}
    where $C := \max\{1, \|\cos t~id-\sin t S\|_{L^\infty(\Omega)}\} \leq \sqrt{n} + \|S\|_{L^\infty(\Omega)}$. Combining this with \eqref{eq:rhs_bd_by_A} and \eqref{eq:Thetasigma_vs_G}, we have shown that
    \begin{equation} \label{eq:rhs_Phi_to_phi}
        \operatorname{dist}^p(d\Phi_{(x,t)}, SO(\Theta^*\sigma,\sigma)) \leq C \big( \operatorname{dist}^p(d\phi_x, O(g,\sigma)) + |d\phi_x \circ (S-S_\phi)|^p_{g,\sigma} \big),
    \end{equation}
    which is our desired upper bound for the right-hand side of \eqref{eq:codim0_prelim}.
    
    Now, we consider the left-hand side of \eqref{eq:codim0_prelim}. We have
    \begin{align*}
        d\Phi - Qd\Theta &= \cos t ~ \big((d\phi -Qd\theta) \oplus (\nu_\phi - Q\nu_\theta)\big) + \sin t ~ \big( (d\nu_\phi - Q d\nu_\theta) \oplus (-\phi + Q\theta)\big)\\
        &:= \cos t ~ T_1 + \sin t ~ T_2
    \end{align*}
    where $T_i: TM \times \mathbb{R} \to \mathbb{R}^{n+2}$ are linear maps which depend only on $x \in M$ (not on $t$). Define
    \begin{equation*}
        \omega_+ := \{ x \in \omega : \langle T_1, T_2 \rangle_{G,e} \geq 0\}, \quad \omega_- := \omega \setminus \omega_+.
    \end{equation*}
    Then,
    \begin{equation*}
        \begin{split}
            \int_{\Omega} &|d\Phi - Q d\Theta|_{G,e}^p dvol_G\\ &= \int_{\omega \times (-h,h)} \left( \cos^2 t ~ |T_1|^2_{G,e} + 2 \cos t \sin t ~ \langle T_1, T_2 \rangle_{G,e}  + \sin^2 t ~ |T_2|^2_{G,e}\right)^{p/2} ~ dvol_G\\
            &\geq \int_{\omega_+ \times(0,h) ~\cup~ \omega_- \times (-h,0)} \left( \cos^2 t ~ |T_1|^2_{G,e} + \sin^2 t ~ |T_2|^2_{G,e}\right)^{p/2} ~ dvol_G\\
            &\geq C \int_\omega \left( |T_1|^2_{G,e} + |T_2|^2_{G,e}\right)^{p/2} ~ dvol_G.
        \end{split}
    \end{equation*}
    In the last line above we integrated out the $t$-direction; $C$ depends only on $h$ and therefore on $\Theta$, $\Omega$, and $S$. Finally, we have
    \begin{align*}
        |T_1|_{G, e}^2 &= |d\phi - Qd\theta|_{g,e}^2 + |\nu_\phi - Q \nu_\theta|_{e}^2\\
        |T_2|_{G,e}^2 &= |d\nu_\phi - Q d\nu_\theta|_{g,e}^2 + |\phi-Q\theta|_{e}^2 
    \end{align*}
    and hence
    \begin{equation} \label{eq:lhs_from_below}
        \int_{\Omega} |d\Phi - Q d\Theta|_{G,e}^p dvol_G \geq C  \big(\|\phi - Q \theta\|_{W^{1,p}(\omega; \mathbb{R}^{n+2})}^p + \|\nu_\phi - Q \nu_\theta\|_{W^{1,p}(\omega;\mathbb{R}^{n+2})}^p\big),
    \end{equation}
    where $C$ depends only on $\Theta$, $M$, $g$, $S$, $n$, and $p$.  
    
    We conclude the proof of Theorem~\ref{thm:rig_main_codim1} from \eqref{eq:codim0_prelim}, \eqref{eq:rhs_Phi_to_phi}, and \eqref{eq:lhs_from_below}.     
\end{proof}

\bigskip

\noindent
{\bf Acknowledgement}. The authors thank Professor Gui-Qiang Chen for insightful discussions. 

The research of SL is supported by NSFC Projects 12331008 $\&$ 12411530065, the Young Elite Scientists Sponsorship Program by CAST 2023QNRC001, National Key Research $\&$ Development Programs 2023YFA1010900 and 2024YFA1014900, Shanghai Rising-Star Program 24QA2703600, Shanghai Qi-Guang Scholarship, and Shanghai Frontiers Science Center of Modern Analysis. The research of IN is supported by a Clarendon scholarship from the University of Oxford.

\medskip
\noindent
{\bf Statement of competing interests}. We declare that there are no conflicts of interest involved.

\medskip
\noindent
{\bf Data Availability Statement}. We declare that no data are associated with this work.

% \medskip
% \noindent
% {\bf AI Statement}.  The author thanks the DeepSeek and ChatGPT 5.5/5.6 AI models for computational assistance during the exploration of geometric conditions and the construction of counterexamples for the metric families considered in this work. All mathematical derivations, conclusions, and errors remain solely the responsibility of the author.
\medskip
\noindent
{\bf AI Statement}. 
ChatGPT 5.6 Sol helped us to make Lemma 6 rigorous and to refine the patching procedure in the proof of Theorem 2. The authors take full responsibility for all content of this manuscript.
% The results of this note were obtained without any AI assistance. However, while preparing the latest version of this manuscript, Chat GPT 5.6 suggested a simplification of the patching procedure in the proof of Theorem 2 and alerted the authors to the references \cite{bastug_asymptotic_noncompact, bastug_quant_local}. The authors take full responsibility for all content of this manuscript.

\bibliography{refs}        %use a bibtex bibliography file refs.bib
\bibliographystyle{siam}  %use the plain bibliography style

\end{document}